\documentclass[10pt,leqno]{amsart}
\usepackage{graphicx}
\usepackage{indentfirst,csquotes}

\usepackage{amssymb,amsthm,amsmath}
\usepackage[indentafter]{titlesec}

\usepackage{xcolor,paralist,hyperref,fancyhdr,etoolbox}
\newtheorem{theorem}{Theorem}[]
\newtheorem{definition}[theorem]{Definition}

\newtheorem{lemma}[theorem]{Lemma}
\newtheorem{proposition}[theorem]{Proposition}
\newtheorem{corollary}[theorem]{Corollary}

\titleformat{\section}[display]{\normalfont\huge\bfseries\centering}{\centering\chaptertitlename\thechapter}{10pt}{\Large}
\titlespacing*{\section}{0pt}{0ex}{0ex}

\hypersetup{ colorlinks=true, linkcolor=black, filecolor=black, urlcolor=black }

\usepackage{lipsum}

\newcommand{\Hom}{\mathrm{Hom}}

\newcommand{\Aut}{\mathrm{Aut}}

\newcommand{\Circ}{\mathrm{C}}
\renewcommand{\section}[1]{%
  \par\bigskip 
  {\centering \normalfont\Large\bfseries #1\par} 
  \medskip 
}

\begin{document}
\title{On Cellular Automata} 
\author[Initial Surname]{Tawfiq Hamed}
\author[Initial Surname]{Mohammad Saleh}
\date{\today}

\email{1185221@student.birzeit.edu}
\email{msaleh@birzeit.edu}
\maketitle

\let\thefootnote\relax

\begin{abstract}
Cellular automata are a fundamental computational model with applications in mathematics, computer science, and physics. In this work, we explore the study of cellular automata to cases where the universe is a group, introducing the concept of \( \phi \)-cellular automata. We establish new theoretical results, including a generalized Uniform Curtis-Hedlund Theorem and linear \( \phi \)-cellular automata. Additionally, we define the covering map for \( \phi \)-cellular automata and investigate its properties. Specifically, we derive results for quotient covers when the universe of the automaton is a circulant graph. This work contributes to the algebraic and topological understanding of cellular automata, paving the way for future exploration of different types of covers and their applications to broader classes of graphs and dynamical systems.  
\end{abstract} 

\bigskip

\section{Introduction}
\bigskip
A cellular automaton is a discrete model of computation, originally developed in the 1940s by Stanislaw Ulam and John von Neumann, as a simplified mathematical model of biological systems.

In this paper, we aim to explore and understand certain properties of cellular automata through the lenses of group theory and graph theory. Our work is inspired by the seminal monograph by Ceccherini-Silberstein and Coornaert \cite{CeccheriniSilberstein2010} and builds upon the contributions of Castillo-Ramirez \cite{Castillo_Ramirez_2023} and \cite{CastilloRamirez2024} on the generalization of cellular automata, as well as several results from \cite{CeccheriniSilberstein2010}. Additionally, we generalize some of the results from \cite{CeccheriniSilberstein2010}, define a covering map for $\phi$-cellular automata, and prove several interesting results related to this concept.

\section{Preliminaries}
\bigskip

A set $G$ with a closed binary operation $*$ is called a \textbf{semigroup} if it satisfies the \textbf{associative law}: for every $x, y, z \in G$, we have $(x * y) * z = x * (y * z)$. If, in addition, there exists an element $e \in G$ (called the \textbf{identity element}) such that for every $x \in G$, $e * x = x * e = x$, then $G$ is called a \textbf{monoid}.

Furthermore, if every element $x \in G$ has an \textbf{inverse} $x^{-1} \in G$ such that $x * x^{-1} = x^{-1} * x = e$, then $G$ is called a \textbf{group}. We say that $G$ is an \textbf{abelian (commutative) group} if $x * y = y * x \in G$ for all $x, y \in G$. Let $H \neq \emptyset$ be a subset of $G$. If $x, y \in H$, then $xy^{-1} \in H$, and thus, $H$ is called a \textbf{subgroup} of $G$.  
Let $G$ be a group and $H$ a subgroup of $G$. For $x \in G$ the set $xH = \{xh | h \in H\}$ called a \textbf{left coset} and the set $Hx = \{hx | h \in H\}$ called a \textbf{right coset}. And we say that $N$ is a \textbf{normal subgroup} of $G$, denoted by $N \trianglelefteq G$, if for every $g \in G$, it holds that $gN = Ng$. The set of left cosets $G/N$ forms a group, which is called the \textbf{factor group} or \textbf{quotient group}.

Let \( \phi: G \to H \) be a group homomorphism. The \textbf{image} of \( \phi \) is defined as
\[
\text{Im}(\phi) = \{ \phi(g) \mid g \in G \}.
\]

The \textbf{kernel} of \( \phi \) is defined as

\[
\text{Ker}(\phi) = \{ g \in G \mid \phi(g) = 1 \}.
\]

Group actions provide a powerful tool for understanding the structure and properties of groups by interpreting the elements of a group as permutations of a set. They play a key role in the definition of cellular automata with groups. In this section, we will define group actions and explore some examples to illustrate the concept.

\begin{definition}
    A group \( G \) is said to \textbf{act} on a nonempty set \( X \) if there exists a map 
    \[
    \phi: G \times X \to X
    \]
    called a \textbf{group action}, satisfying the following conditions for all \( x \in X \) and \( g, h \in G \): 
    \begin{itemize}
        \item \(\phi(e, x) = x\), where \( e \) is the identity element of \( G \),
        \item \(\phi(g, \phi(h, x)) = \phi(gh, x)\), or equivalently, \( g(hx) = (gh)x \).
    \end{itemize}
    In this case, \( X \) is called a \textbf{\( G \)-set}.
\end{definition}

Let \( \Aut(X) \) (or \( Sym(X) \)) denote the set of all bijections from \( X \) to \( X \). We can observe that \( \Aut(X) \) forms a group under composition. Thus, a group action can be defined as a group homomorphism 
\[
\phi: G \to \Aut(X).
\]

Let $G$ be a group, and let $S \subset G$ be a generating set of $G$ such that $1 \notin S$ and $S = S^{-1}$. Then, the pair $(G, S)$ is called an undirected \textbf{Cayley graph}, denoted by $\Gamma$, where G is the vertex set and $(g,g')$ is an edge if and only if there exists $s \in S$ such that $g' = gs$. And will assume all graphs are \textbf{finite}, \textbf{connected}, \textbf{free loop}, and \textbf{without multiple edges}.

\begin{proposition}\cite[Proposition 10]{Oggier}\label{qg-cyclic}
The quotient of a cyclic group $G$ is cyclic.
\end{proposition}

\begin{theorem}\cite[Theorem 4.1]{Bhattacharya1994}\label{cyclic-iso-zn}
    Every cyclic group of order n is isomorphic to $\mathbb{Z}_n$ for some $n \in \mathbb{N}.$
\end{theorem}

\begin{definition}\cite{toomey2014}
If $G$ is a cyclic group, then the Cayley graph of $G$ defined on some connection set $S$ is called a \textbf{circulant}. We will denote these graphs $\Circ_G$.
\end{definition}

\begin{definition}\label{def-iso}
Let $G,H$ be a graph, the \textbf{graph homomorphism} $\psi: G \rightarrow H$ is a function from $V(G)$ to $V(H)$, such that for any $(u, v) \in E(G)$ implies $(\psi(u), \psi(v)) \in E(H)$ for all $u,v \in V(G)$. If $\psi$ is bijective then $\psi$ is \textbf{graph isomorphism}.
\end{definition}

\begin{definition}\label{graph-hom}
Let $G, H$ be a graph, and there exists a graph homomorphism $\psi: G \rightarrow H$, if $\psi$ is subjective, then $\psi$ is a \textbf{covering map}, and it's \textbf{n-fold} if $\psi$ is n-to-one.
\end{definition}

A topology on a set $X$ is a collection of subsets of $X$, called open sets, that satisfies certain axioms. The smallest topology on $X$, known as the trivial topology, consists only of the empty set $\emptyset$ and $X$. The largest topology on $X$, called the \textbf{discrete topology}, consists of all subsets of $X$ as open sets.

 When each $X_\lambda$ is endowed with the \textbf{discrete topology}, the \textbf{product topology} on $X = \prod_{\lambda \in \Lambda} X_\lambda$ is called the \textbf{prodiscrete topology}.\cite{CeccheriniSilberstein2010}

\begin{proposition}\cite[Proposition B.2.2]{CeccheriniSilberstein2010}\label{uf-con}
Let $X$ and $Y$ be uniform spaces associated with the uniform structures. Then every uniformly continuous map $f: X \rightarrow Y$ is continuous.
\end{proposition}

\begin{definition}\cite{Angluin1981}\label{def-cover}
    Let G and H are graphs and covering map $\psi: G \to H$, then G is a \textbf{covering} of H via $\psi$ if for every $v \in V(G)$, $\psi$ induces a one-to-one correspondence between the vertices adjacent to v in G and the vertices adjacent to $\psi(v)$ in H.
\end{definition}

\section{Cellular Automata and Groups}
\bigskip
In this section, we explore cellular automata through the terminology and concepts of group theory. This section is based on \cite{CeccheriniSilberstein2010}.

In the next definition, we referred to the set \( A \) as the \textbf{alphabet}, with the elements of \( A \) representing the states of the cells in the cellular automaton. These elements are often called \textbf{letters}, \textbf{states}, \textbf{symbols}, or \textbf{colors}. Examples of states include \( \{0, 1\} \), \( \{\text{dead, live}\} \), etc. Throughout this work, we assume that $A$ is finite and denote the cardinality of the set $A$ by $|A|$, with $|A| > 1$.

The group \( G \) is called the \textbf{universe}, with the elements of \( G \) representing the cells in the automaton or representing the vertices of the graph when the cellular automaton is defined over a Cayley graph. We assume \( G \) is a countable group.

Let $\Omega$ be a finite subset of $G$, then the map $p: \Omega \rightarrow G$ is called a \textbf{pattern} (or \textbf{block}) over the group $G$, and the set $\Omega$ is called the \textbf{ support} of $p$.

The set for all maps from $G$ to $A$ is called the set of \textbf{conﬁgurations}, defined by $A^G$ is given by

\[ A^G = \prod_{g \in G} A = \{x: G \to A\}.\]

The \textbf{left action} of $G$ on $A^G$ is given by the map
\[ \psi: G \times A^G \rightarrow A^G \]
defined by
\[ (g,x) \mapsto gx \]
where $g \in G$ and a configuration $x \in A^G$. This action is called the \textbf{G-shift} on $A^G$. 

\begin{definition}\cite[Definition 1.4.1]{CeccheriniSilberstein2010}\label{ca-def}
Let G be a group and let A be a set, a cellular automaton over the group G and the alphabet A is a map $\tau: A^{G} \rightarrow A^{G}$ satisfying the following property: there exists a finite subset $S \subset G$ and a map $\mu: A^{S} \rightarrow A$ such that
\[ \tau(x)(g) = \mu((g^{-1}x)|_S) \]
for all $x \in A^G$ and $g \in G$, where $(g^{-1}x)|_S$ denotes the restriction of the configurations $g^{-1}x$ to S.
Such a set $S$ is called a memory set and $\mu$ is called a local deﬁning map
for $\tau$.
\end{definition}

The set \( S \) is called the memory set of \( \tau \), also referred to as the set of neighborhoods or the window in other literature (For example, let's defined the window as \( 2N + 1 \) if \( N = 1 \), then the window has size 3. For \( i = 0 \), the window is \( \{-1, 0, 1\} \)). 

By definition \ref{ca-def}, the map $p: S \to A$ is a pattern, and the set of all patterns over $S$ is defined by 
\[ A^S = \{p : S \to A\}. \]

It follows that we can define a map $\mu: A^S \to A$, which is called a local map, also known as a block map. By definition \ref{ca-def}, for $x \in A^G$, applying $\tau$ to the configuration $x$ is equivalent to applying the local map $\mu: A^S \to A$ homogeneously and in parallel, using the left action of $G$ on $A^G$

\begin{proposition}\cite[Proposition 1.2.2]{CeccheriniSilberstein2010}\label{act-con}
    The action of G on $A^G$ is continuous.
\end{proposition}

Other interesting properties of the space \( A^G \) are that it is Hausdorff and totally disconnected, as will be shown in the next proposition. Additionally, for cellular automata over finite alphabets \( A \), there is a very useful topological property: since the product of finite spaces is compact, \( A^G \) is also compact.

\begin{proposition}\cite[Proposition 1.2.1]{CeccheriniSilberstein2010}\label{ag-prop}
    Let \( G \) be a group and \( A \) a set of alphabet. Then the following properties hold for the space \( A^G \):
    \begin{enumerate}
        \item \( A^G \) is Hausdorff.
        \item \( A^G \) is totally disconnected.
        \item \( A^G \) is compact if \( A \) is finite.
    \end{enumerate}
\end{proposition}

\section{$\phi$-Cellular Automaton}
\bigskip
In this section, we present the main results from Castillo-Ramírez's works \cite{Castillo_Ramirez_2023} and \cite{CastilloRamirez2024}, and extend some results from \cite{CeccheriniSilberstein2010}, building upon their definitions and results.

As we saw in Definition \ref{ca-def}, the cellular automaton \( \tau \) is a map from some shift space \( A^G \) to the same space. The question is, if we have two discrete spaces with different groups, and these groups have a homomorphism, then we can define a new map between the two prodiscrete spaces or the shift spaces. How can we define this map? What are the properties of this map? In this section, we will explore the answers to these questions.

\begin{definition}\cite{Castillo_Ramirez_2023}\label{gca-def}
For any $\phi \in \Hom(H,G)$, a $\phi$-\emph{cellular automaton} from $A^G$ to $A^H$ is a function $\tau : A^G \to A^H$ such that there is a finite subset $S \subseteq G$, called a \emph{memory set} of $\tau$, and a \emph{local function} $\mu : A^S \to A$ satisfying
    \[ \tau(x)(h) = \mu (( \phi(h^{-1}) \cdot x) \vert_{S}),  \quad \forall x \in A^G, h \in H.  \]
\end{definition}

We can observe that Definition \ref{ca-def} is equivalent to Definition \ref{gca-def} when $\phi = \text{Id}_G$. Thus, an $\text{Id}_G$-cellular automaton corresponds to the classical cellular automaton.

\begin{lemma}\cite[Lemma 1]{Castillo_Ramirez_2023}\label{le-equivariance} 
Let $G,H$ be a groups, $A$ be a set, and $\phi \in \Hom(H,G)$, Then every $\phi$-cellular automaton is $\phi$-equivariant.
\end{lemma}

\begin{lemma}\cite[Lemma 2]{Castillo_Ramirez_2023}\label{le-continuity}
        Let $G,H$ be a groups, $A$ be a set, and $\phi \in \Hom(H,G)$, Then every $\phi$-cellular automaton is continuous.
\end{lemma}

\begin{lemma}\cite[Lemma 3]{Castillo_Ramirez_2023}\label{le-constant}
Let $S$ be a finite subset of $G$ and $\phi \in \Hom(H,G)$. A $\phi$-equivariant function $\tau : A^G \to A^H$ is a $\phi$-cellular automaton with memory set $S$ if and only if the function $p_{e_G} \circ \tau : A^G \to A$ is constant on $V(x,S)$, for every $x \in A^G$.
\end{lemma}

As shown in \cite{Castillo_Ramirez_2023}, we can construct a $\phi$-cellular automaton, as it demonstrates that for every $\phi \in \Hom(H,G)$, we can define $\phi^{\star} : A^G \to A^H$ by
\begin{align}\label{star}
     \phi^\star(x) = x \circ \phi, \quad \forall x \in A^G.
\end{align}
So, we can construct \( \phi^\star \) as a \( \phi \)-cellular automaton with the memory set \( S = \{ e_G \} \) and the local function \( \mu = id_A \). It is observed that for any \( h \in H \),
\[
\phi^\star(x)(h) = x \circ \phi(h) = (\phi(h^{-1}) \cdot x)(e_G),
\]
where \( e_G \) is the identity element of \( G \).

\begin{lemma}\cite[Lemma 2]{CastilloRamirez2024}\label{le-star}
Let \( G \) and \( H \) be groups, and let \( \phi \in \Hom(H, G) \). Then the following holds:
\begin{enumerate}
\item Let $\psi \in \Hom(H,G)$, then
\[ \phi^{\ast}=\psi^{\ast} \quad \Leftrightarrow \quad \phi=\psi. \]
\item For any group $K$ and any $\psi \in \Hom(K,H)$, we have
\[ (\psi \circ \phi)^* = \phi^* \circ \psi^*. \]
\item $\phi$ is surjective if and only if $\phi^{\ast}$ is injective.
\item $\phi$ is injective if and only if $\phi^{\ast}$ is surjective.
\end{enumerate}
\end{lemma}

\section{Decomposition of $\phi$-cellular automata}
\bigskip
In \cite[Theorem 2]{Castillo_Ramirez_2023} proved that the composition of two generalized cellular automata is also generalized cellular automata, so, the question arises: if we have a \( \phi \)-cellular automaton where \( \phi \in \Hom(H, G) \), can it be decomposed into two generalized cellular automata?

\begin{proposition}\cite[Corollary 12.4]{Gorodentsev2016}\label{de-hom}(Decomposition of Homomorphisms)
    A group homomorphism $\phi: G \to H$ can be factored as a composition of a quotient epimorphism
    \[ G \to G/\ker(\phi)\]
    followed by an injective homomorphism
    \[ G/\ker(\phi) \to H \]
    sending the coset $g \ker(\phi) \in G/\ker(\phi)$ to $\phi(g) \in H$. In particular, $Im(\phi) \cong G/\ker(\phi)$.
\end{proposition}

By this proposition we can conclude the first result in our work

\begin{theorem}
    A \( \phi \)-cellular automaton, where \( \phi \in \Hom(H, G) \), can be factored as a composition of two generalized cellular automata.
\end{theorem}
\begin{proof}
    Let \( \tau: A^G \to A^H \) be a $\phi$-cellular automaton where $\phi \in \Hom(H,G)$, and assume $\ker(\phi) = K$. By Decomposition of Homomorphisms (Proposition \ref{de-hom}), we can factor \( \phi \) to 
    \[ \psi_1: H \to H/K \]
    and \[ \psi_2: H/K \to G \]
    such that 
    \[ \phi = \psi_2 \circ \psi_1 \].

    Now, consider a map \( \psi_1^* = x \circ \psi_1 \) for $x \in A^H$, and \( \psi_2^* = y \circ \psi_2 \) for $y \in A^{H/K}$. As we know by equation \ref{star} the \( \psi_1^* \) is a \( \psi_1 \)-cellular automaton, and the \( \psi_2^* \) is a \( \psi_2 \)-cellular automaton.

    However, \( \psi_1^* \circ \psi_2^* = (\psi_2 \circ \psi_1)^* \) is a \((\psi_2 \circ \psi_1) \)-cellular automaton by \cite[Theorem 2]{Castillo_Ramirez_2023}. Since \( \phi = \psi_2 \circ \psi_1 \), we can conclude by lemma \ref{le-star} \((\psi_2 \circ \psi_1)^* = \tau\).
\end{proof}

\section{Uniform Curtis–Hedlund Theorem}
\bigskip
The Curtis–Hedlund theorem has been generalized for prodiscrete uniform structures in \cite{CeccheriniSilberstein2010} and \cite{CECCHERINISILBERSTEIN2008225}. Now, we further extend this generalization to $\phi$-cellular automata, where $\tau: A^G \rightarrow A^H$ and $\phi: H \rightarrow G$.

\begin{theorem}\label{Curtis-Hedlund-Uniformly}
Let A be a set, G and H be groups, and $\phi \in \Hom(H,G)$. Let $\tau: A^G \rightarrow A^H$ be a map and equip $A^G$ with its prodiscrete uniform structure. Then the following conditions are equivalent:
\begin{itemize}
    \item $\tau$ is a $\phi$-cellular automaton.
    \item $\tau$ is uniformly continuous and $\phi$-equivariant.
\end{itemize}
\end{theorem}
\begin{proof}
    First, we will show that $\tau$ is uniformly continuous.
    Let $\tau: A^G \rightarrow A^H$ is a $\phi$-cellular automaton and $\phi \in \Hom(H,G)$. Let $S \subset G$ be a memory set, and $\mu: A^S \rightarrow A$ is a local map. Now, for $x,y \in A^G$, let $h \in \Omega \subset H$,we have:
    \begin{eqnarray*}
        \tau(x)(h) &=& \mu (( \phi(h^{-1}) \cdot x) \vert_{S}) \\
        &=&  \mu (( x(\phi(h))) \vert_{S}) \\
        &=&  \mu (( y(\phi(h))) \vert_{S}) \\
        &=& \mu (( \phi(h^{-1}) \cdot y) \vert_{S}) \\
        &=& \tau(y)(h).
    \end{eqnarray*}
    So, $x$ and $y$ coincide on $\Omega S = \{\phi(h)s : h \in \Omega, s \in S\}$, it follows that $\tau(x)$ and $\tau(y)$ coincide on $\Omega$. Then we deduce that:
    \[ ( \tau \times \tau )(W_{\Omega S}) \subset W_{\Omega} \]
    for every finite subset $\Omega$ of $H$. As the sets $W_\Omega$, where $\Omega$ runs over all finite subsets of $H$, form a base of entourages for the prodiscrete uniform structure on $A^H$, it follows that $\tau$ is uniformly continuous and it $\phi$-equivariant by Lemma \ref{le-equivariance}. \\
    Conversely, suppose that $\tau$ is uniformly continuous and $\phi$-equivariant. Consider the subset $\Omega = \{1_H\} \subset H$. Since $\tau$ is uniformly continuous, there exists a finite subset $S \subset G$, such that 
    \[ ( \tau \times \tau )(W_{S}) \subset W_{\Omega}, \]
    for all $x \in A^G$. Using the $\phi$-equivariant of $\tau$, we get
    \begin{eqnarray*}
        \tau(x)(h) &=& [ h^{-1} \tau(x) ](1_G) \\
        &=& \tau(\phi(h^{-1})x)(1_G) \\
        &=&  \mu (( x(\phi(h))) \vert_{S}), \\
    \end{eqnarray*}
     for all $x \in A^G$ and $h \in H$.
     This shows that $\tau$ is a $\phi$-cellular automaton with memory set $S$.
\end{proof}

\section{Linear $\phi$-Cellular Automaton}
\bigskip
In this section, we extend the linear version of the Curtis–Hedlund theorem \cite[Theorem 8.1.2]{CeccheriniSilberstein2010}
to $\phi$-Cellular Automaton in case where the alphabet is a vector space, as well as \cite[Proposition 8.1.1]{CeccheriniSilberstein2010}.

 Let $G$ be a group, $V$ be a vector space over a field $\mathbb{K}$, and $\phi$ be a homomorphism from $H$ to $G$.

\begin{definition}\label{lca-phi-def}
For any $\phi \in \Hom(H,G)$, a linear $\phi$-cellular automaton over the group $G$ and the alphabet $V$ is a cellular automaton $\tau: V^G \rightarrow V^H$ which is $\mathbb{K}$-Linear, i.e., which satisﬁes

\[ \tau(x + x') = \tau(x) + \tau(x') \] and \[ \tau(kx) = k\tau(x) \]
for all $x,x' \in V^G$ and $k \in \mathbb{K}$.
\end{definition}

By the previous definition \ref{lca-phi-def} we will extend \cite[Proposition 8.1.1]{CeccheriniSilberstein2010} to linear $\phi$-cellular automaton.

\begin{proposition}\label{lca-proposition}
For any $\phi \in \Hom(H,G)$, let $G,H$ be a groups and let $V$ be a vector space over a ﬁeld $\mathbb{K}$. Let $\tau: V^G \rightarrow V^H$ be a $\phi$-cellular automaton with memory set $S \subset G$ and local deﬁning map $\mu: V^S \rightarrow V$ Then $\tau$ is linear if and only if $\mu$ is $\mathbb{K}$-linear.
\end{proposition}
\begin{proof}
    The first direction is trivial and follows the same proof as in \cite[Proposition 8.1.1]{CeccheriniSilberstein2010}.
    
    Now, let us prove the converse. Suppose that $\mu$ is $\mathbb{K}$-linear. Then for all $x,y \in V^G$, $k \in \mathbb{K}$ and $h \in H$, we have
    \begin{eqnarray*}
        \tau(x + y)(h) &=& \mu((\phi(h^{-1})(x + y))|_S) \\
        &=& \mu((\phi(h^{-1})x)|_S + (\phi(h^{-1})y)|_S) \\
        &=& \mu((\phi(h^{-1})x)|_S) + \mu((\phi(h^{-1})y)|_S) \\
        &=& \tau(x)(h) + \tau(y)(h) \\
        &=& (\tau(x) + \tau(y))(h)
    \end{eqnarray*}
    and 
    \begin{eqnarray*}
        \tau(kx)(h) &=& \mu((\phi(h^{-1})(kx)|_S) \\
        &=& \mu(k(\phi(h^{-1})x)|_S) \\
        &=& k\mu((\phi(h^{-1})x)|_S) \\
        &=& k\tau(x)(h)
    \end{eqnarray*}
    It follows that $\tau$ is linear.
\end{proof}

In this work, we generalize the first part of \cite[Theorem A]{shahryari2019notecellularautomata} to leverage it in our proofs.

\begin{lemma}\label{f-tao-uc}
    For any $\tau: A^G \to A^H$, define a map $f_\tau: A^G \to A$ by 
    \[ f_\tau(x) = \tau(x)(1) \]
    where $x\in A^G$, then $f_\tau$ is uniformly continuous.
\end{lemma}
\begin{proof}
    Let $G, H$ be a groups, $\phi \in \Hom(H,G)$, and let $\mathcal{U}$ be a prodiscreate uniform structure over $A^G$. For a memory set $\Omega \subset G$, we have $W_\Omega \in \mathcal{U}$ where $W_\Omega = \{(x,y) \in A^G \times A^G : x|_\Omega = y|_\Omega\}$.
    Consider a map $f_\tau \times f_\tau \mapsto A \times A$ defined by
    \[ (f_\tau \times f_\tau)(x,y) = (f_\tau(x),f_\tau(y)). \]
    That implies 
    \[ x|_\Omega = y|_\Omega \Rightarrow \tau(x)(1) = \tau(y)(1). \]
    That's mean $W_\Omega \subseteq (f_\tau \times f_\tau )^{-1}(\Delta_A) \in \mathcal{U}$.
    It follows that $f_\tau$ is uniformly continuous.
\end{proof}

The next result extends the linear version of the Curtis–Hedlund theorem \cite[Proposition 8.1.2]{CeccheriniSilberstein2010}.

\begin{theorem}\label{lca-cht}
    For any $\phi \in \Hom(H,G)$, let $G$ be a group and let $V$ be a vector space over a ﬁeld $\mathbb{K}$. Let $\tau: V^G \rightarrow V^H$ be a $\phi$-equivariant and $\mathbb{K}$-linear map. Then the following conditions are equivalent:
    \begin{itemize}
        \item the map $\tau$ is a linear cellular automaton;
        \item the map $\tau$ is uniformly continuous (with respect to the prodiscrete uniform structure on $V^G$ and $V^H$);
        \item the map $\tau$ is continuous (with respect to the prodiscrete topology on $V^G$ and $V^H$);
        \item the map $\tau$ is continuous (with respect to the prodiscrete topology on $V^G$ and $V^H$) at the constant conﬁguration $x = 0$.
    \end{itemize}
\end{theorem}
\begin{proof}
    (a) $\Rightarrow$ (b): This follows immediately from Theorem \ref{Curtis-Hedlund-Uniformly}.  

(b) $\Rightarrow$ (c): This also follows immediately, as the topologies on $A^G$ and $A^H$ are associated with the prodiscrete uniform structure. Since every uniformly continuous map is continuous, it follows from Proposition \ref{uf-con} that $\tau$ is continuous.  

(c) $\Rightarrow$ (d): This implication is trivial.  

(d) $\Rightarrow$ (a): Suppose that $\tau$ is continuous at 0. Then By Lemma \ref{f-tao-uc} and Proposition \ref{uf-con} the map $f_{\tau}: A^G \to A$ is continuous. We deduce that there exists finite $M \subset G$ such that if $x \in V^G$ satisfies $x(m) = 0$ for all $m \in M$, then $\tau(x)(1) = 0$. Then the map $f_{\tau}$ is constant on $V(x,M)$, for every $x \in A^G$, by Lemma \ref{le-constant} we conclude that $\tau$ is $\phi$-cellular automaton with memory set M.
Thus, (d) implies (a).

\end{proof}

\section{Covering $\phi$-Cellular Automaton}
\bigskip
Studying covering maps is important in distributed computing, graph theory, and algebraic topology. Notable works in these areas include \cite{Angluin1980} in distributed computing, \cite{SDS2008} in sequential dynamical systems, and \cite{Sunada2013} in topological crystallography. 

In this section, we will explore specific cases of covering maps for cellular automata. When we began working on this thesis, we initially considered \( A^G \) as a covering space. However, by Proposition \ref{ag-prop}, we found that \( A^G \) is a totally disconnected and not path-connected space. Therefore, we approached the concept of covering maps from a different perspective and identified \(\phi\)-cellular automata as a promising starting point.

Moreover, we will study a \( \phi \)-cellular automaton over a circulant graph, which is a Cayley graph for a cyclic group. We choose this starting point because this graph can be viewed as a group, as the set of vertices forms a cyclic group. Thus, we can generalize all previous results of \( \phi \)-cellular automata over groups to this setting. We hope this work serves as a foundation or inspiration for studying other types of graphs generated from groups and their covering maps.

The quotient map is a covering map, and this map suits our purpose because it can be viewed both as a group homomorphism and as a graph homomorphism, specifically a covering map. Since our case involves Cayley graphs where the set of vertices forms a cyclic group, we can study a covering of \( \phi \)-cellular automaton using this framework. However, it is important to note that not every graph homomorphism is a group homomorphism.

\begin{lemma}\label{le-cover-group}
Let \( G \) be a cyclic group, and let \( N \trianglelefteq G \), where \( G/N \) is a quotient group of \( G \). Then the quotient map \( \psi: G \to G/N \) is a graph homomorphism, and in particular, a covering map from \( \Circ_G \) to \( \Circ_{G/N} \).

\end{lemma}
\begin{proof}
Let \( G \) be a cyclic group and let \( N \trianglelefteq G \) such that \( G/N \) is a quotient group of \( G \). Now, for the quotient map \( \psi: G \rightarrow G/N \), we can clearly see that it is a homomorphism:  
\[
\psi(x)\psi(y) = xN \cdot yN = xyNN = xyN = \psi(xy),
\]  
for every \( x, y \in G \).  

Since \( G/N \) is cyclic by Proposition \ref{qg-cyclic}, we can consider the circulant graphs \( \Circ_{G/N} \) and \( \Circ_G \), where \( G/N \) and \( G \) are the sets of vertices, respectively.  

This means, by Definition \ref{graph-hom}, that \( \psi \) is a graph homomorphism. Furthermore, since \( \psi \) is surjective, it follows that \( \psi \) is a covering map.
\end{proof}

\begin{corollary}
    A covering map $\psi: \Circ_G \rightarrow \Circ_{G/N}$ is $|N|$-fold.
\end{corollary}
\begin{proof}
    This is trivial by definition \ref{graph-hom}.
\end{proof}

\begin{theorem}
    Let \( \Circ_G \) and \( \Circ_{H} \) be a circulant graphs, If \( \Circ_{H} \) cover \( \Circ_G \) by the covering map $\phi$, then there exist an injective \( \phi \)-cellular automaton.
\end{theorem}
\begin{proof}
    Let \( \phi: \Circ_{H} \to \Circ_G \) be a covering map, As we know from the Lemma \ref{le-cover-group} the $\phi$ is a group homomorphism, then we can consider \( \phi^* = x \circ \phi \) for $x \in A^{\Circ_{G}}$, It follows that \( \phi^* \) is a \( \phi \)-cellular automaton by equation \ref{star}. Also, since \( \phi \) is surjective, it follows by Lemma \ref{le-star} that \( \phi^* \) is injective.

\end{proof}

\section{Future Work}
\bigskip
Cellular automata are a promising topic, and there are still many questions and areas that need further exploration. In future work, we plan to study the algebraic and topological properties of the covering map of \( \phi \)-cellular automata.

Let \( \tau: A^G \to A^H \) be a \( \phi \)-cellular automaton where \( G \) covers \( X \) and \( H \) covers \( Y \). Can there exist a \( \psi \)-cellular automaton \( \hat{\tau}: A^X \to A^Y \) that covers \( \tau \)? Alternatively, is there a map \( \tau \to \hat{\tau} \), and what are the properties of this map?

Additionally, we plan to study the case where \( G \) is a topological group and investigate the covering properties of the cellular automaton. In this work, we focus on a special case of covering for circulant graphs. We aim to extend the results to other types of graphs and study other types of covers, such as double covers, circulant double covers, finite common covers, and universal covers.

Furthermore, A cellular automaton with different alphabets, which we plan to analyze in depth. Understanding how the choice of alphabets affects the properties of cellular automata and their covers will be an important direction for future research.

$\,$

$\,$

\end{document}